\documentclass[11pt,a4paper]{amsart}
\usepackage{amscd}
\usepackage{comment}
\usepackage{oldgerm}
\usepackage{amsmath}
\usepackage{amssymb}
\usepackage{amsthm}
\usepackage[svgnames]{xcolor}
\usepackage{bm}
\usepackage[colorlinks,linktocpage=true,citecolor=blue,linkcolor=blue,urlcolor=blue]{hyperref}
\usepackage{graphicx}
\usepackage[all]{xy}
\usepackage[foot]{amsaddr}

\makeatletter
\@addtoreset{equation}{section}
\makeatother

\theoremstyle{definition}
\newtheorem{defn}[equation]{Definition}
\newtheorem{defn'}[equation]{Definition'}
\theoremstyle{plain}
\newtheorem{thm}[equation]{Theorem}

\newtheorem{prop}[equation]{Proposition}

\newtheorem{conj}[equation]{Conjecture}
\theoremstyle{remark}

\newcommand{\Z}{\mathbb{Z}}

\newcommand{\R}{\mathbb{R}}
\newcommand{\C}{\mathbb{C}}

\newcommand{\pt}{\mathrm{pt}}

\newcommand{\del}{\partial}

\newcommand{\dR}{\mathrm{dR}}
\newcommand{\Bord}{\mathrm{Bord}}
\newcommand{\cw}{\mathrm{cw}}
\newcommand{\hBord}{h\Bord^{G_\nabla}}

\newcommand{\HS}{\mathrm{HS}}
\newcommand{\CS}{\mathrm{CS}}

\newcommand{\clo}{\mathrm{clo}}
\newcommand{\FL}{\mathrm{FL}}

\usepackage[shortlabels]{enumitem}

\DeclareMathOperator{\Hom}{Hom}

\def\beq#1\eeq{\begin{align}#1\end{align}}
\begin{document}

\title[Invertible QFTs and differential Anderson duals]{Invertible QFTs and differential Anderson duals}
\author[M. Yamashita]{Mayuko Yamashita}
\address{Department of Mathematics, Kyoto University, 
606-8502, Kyoto, Japan}
\email{yamashita.mayuko.2n@kyoto-u.ac.jp}
\subjclass[]{}
\maketitle

\begin{abstract}
This is the proceeding of a talk given at Stringmath 2022. 
We introduce a Cheeger-Simons type model for the differential extension of Anderson dual to generalized homology theory with physical interpretations. 
This construction generalizes the construction of the differential Anderson dual to bordism homology theories, given in a previous work of Yonekura and the author \cite{YamashitaYonekura2021}. 
\end{abstract}

\section{Introduction}

In this article, we introduce a Cheeger-Simons type model for the differential extension of Anderson dual to generalized homology theory with physical interpretations. 
This construction generalizes the construction of the differential Anderson dual to bordism homology theories, given in a previous work of Yonekura and the author \cite{YamashitaYonekura2021}. 

The Anderson duals have been important in the recent study of quantum field theories (QFTs) in terms of algebraic topology. 
The Anderson duals to bordism homology theories are conjectured to classify invertible QFTs by Freed and Hopkins \cite{Freed:2016rqq}. 
Based on this conjecture, Yonekura and the author gave a model $I\Omega^G_\dR$ for the Anderson dual to $G$-bordism homology theory with a physical interpretation \cite{YamashitaYonekura2021}. 
Interestingly, it turns out that the natural object produced from QFTs is a {\it differential refinement} of it, and passing to topological group corresponds to taking deformation classes. 
The construction has a similarity to the Cheeger-Simons' model of differential ordinary cohomology in terms of {\it differential characters} \cite{CheegerSimonsDiffChar}. 

In this article, we explain this construction in a generalized form. 
We provide a Cheeger-Simons' type models of the Anderson dual to generalized homology theories. 
More precisely, we construct a differential extension $\widehat{IE}^*_\dR$ of the Anderson dual $IE^*$ to a spectrum $E$, provided that a differential {\it homology} $\widehat{E}_*$ is given. 

Differential cohomology theories have been much studied, but it seems that the homology version has not been interested so much. 
Elmrabty \cite{ElmrabtyDiffKhom} defines differential $K$-homology, and as far as the author know it is the only example existing explicitly in the literature. 
But certainly there are situations where we want to refine generalized homology in a differential way, for example in the case of bordism homology theories. 
Such situations also naturally arise in physics; for example D-branes fits into the framework of differential $K$-homology. 
Differential bordism theories naturally arise in (non-topological) QFT's. 
Thus the author feels that it is valuable to have a general framework for differential refinements in the homology settings, as we do in Subsection \ref{sec_diffhom}. 

Given a differential homology theory $\widehat{E}_*$, the corresponding differential Anderson dual cohomology $\widehat{IE}^*_\dR$ consists of pairs $(\omega, h)$, where $h \colon \widehat{E}_{n-1} \to \R/\Z$ is a group homomorphism which satisfies a compatibility with a closed differential form $\omega \in \Omega^n(X; E^\bullet(\pt) \otimes \R)$. 
Applied to $E = MTG$, we recover the model $\widehat{I\Omega^G_\dR}$ in \cite{YamashitaYonekura2021}, where
the homomorphism $h$ is regarded as the complex phase of the {\it partition function} of an invertible QFT. 
Applied to $E = H\Z$, we recover the Cheeger-Simons' differential character groups (note that we have the Anderson self-duality $IH\Z \simeq H\Z$). 

This article is organized as follows. 
The first two sections are devoted to preliminaries. 
We give a brief introduction to differential cohomology theories in Section \ref{sec_diffcoh}. We explain invertible QFTs and Anderson duals in Section \ref{sec_QFT}. 
In Section \ref{sec_diffhom}, we introduce differential homology theories, and in Section \ref{sec_IE} we construct the differential Anderson duals. 
In Section \ref{sec_interpretation} we explain the physical interpretation of our model. 

\subsection{Notations and Conventions}
In this article we use differential forms and currents. 
\begin{itemize}
    \item $\Omega^*(X) = \Omega^*(X; \R)$ denotes the space of differential forms on $X$ with coefficients in $\R$. 
    $\Omega_*(X) := \Hom_{\mathrm{conti}} (\Omega^{*}(X), \R)$ denotes the space of compactly supported currents on $X$. 
    \item If $V_\bullet$ is a graded vector space over $\R$, we topologize it as the direct limit of all finite subspaces and define $\Omega^*(X; V_\bullet)$. 
    We also set
    \begin{align*}
        \Omega_n(X; V_\bullet) := \oplus_{p+q = n} \Hom_{\mathrm{conti}} (\Omega^{p}(X),V_q). 
    \end{align*}
    \item The de Rham chain complex is denoted as
    \begin{align*}
        \cdots \xrightarrow{\del} \Omega_{n+1}(X; V_\bullet) \xrightarrow{\del} \Omega_{n}(X; V_\bullet)\xrightarrow{\del}  \cdots. 
    \end{align*}
    \item For a spectrum $E$, we set $V_\bullet^E = E_\bullet(\pt) \otimes \R$ and $V^\bullet_E := E^\bullet(\pt) \otimes \R$, so that $V_n^E = V^{-n}_E$. 
    The homology Chern-Dold character for $E_*$ is denoted by
    \begin{align*}
        \mathrm{ch} \colon E_*(-) \to H_*(-; V_\bullet^E) . 
    \end{align*}
  
\end{itemize}

\section{Background : Differential cohomology}\label{sec_diffcoh}

In this section we give a brief review of generalized differential cohomology theories. 
A {\it differential extension} of a generalized cohomology theory $E^*$ is a refinement $\widehat{E}^*$ of the restriction of $E^*$ to the category of smooth manifolds, which containes differential-geometric data. 

\subsection{Ordinary differential cohomology}

\subsection{$\widehat{H}^2(X; \Z)$}

In order to illustrate the idea of differential cohomology, first we explain a model of the second ordinary differential cohomology. 
Recall that $H^2(X; \Z)$ is identified as the group of isomorphism classes of complex line bundles over a space $X$. 
The differential group is defined by adding differential information on it. 
\begin{defn}\label{def_hatH2}
    For a manifold $X$, we define $\widehat{H}_{\mathrm{geom}}^2(X; \Z)$ to be the abelian group of isomorphism classes of hermitian line bundles with unitary connections over $X$. 
\end{defn}

The differential group is a refinement of the topological group, in the sense that there is a surjection (forgetful map)
\begin{align}
    I \colon \widehat{H}_{\mathrm{geom}}^2(X; \Z) \to {H}^2(X; \Z), \quad [L, \nabla] \mapsto [L]. 
\end{align}
But $\widehat{H}_{\mathrm{geom}}^2$ has more information than $H^2$. We can extract the differential information by taking the curvature, 
\begin{align}
    R \colon \widehat{H}_{\mathrm{geom}}^2(X; \Z) \to \Omega^2_\clo(X), \quad [L, \nabla] \mapsto F_\nabla/(2\pi \sqrt{-1}). 
\end{align}
Connections on topologically trivial line bundles are identified with one form. This gives a map
\begin{align}
    a \colon \Omega^1(X) \to \widehat{H}_{\mathrm{geom}}^2(X; \Z), \quad \alpha \mapsto [\underline{\C}, d+ 2\pi \sqrt{-1}\alpha]. 
\end{align}
These maps are part of the {\it differential cohomology hexagon}, 
\begin{align*}
        \xymatrix@R = 6pt@C=10pt{
        0 \ar[rd] & & & & 0 \\
        & H^{1}(X; \R/\Z)\ar[rr]^-{\mathrm{Bock}} \ar[rd] & & H^2(X; \Z) \ar[ru] \ar[rd]^-{\otimes \R}& \\
        H^{1}(X; \R)\ar[ru] \ar[rd] & & {\widehat{H}_{\mathrm{geom}}^2(X; \Z)}\ar[ru]^-{I} \ar[rd]^-{{R}}& & H^2(X; \R) \\
        & \Omega^{1}(X) / \Omega^{1}_\clo(X)_\Z \ar[ru]^-{{a}} \ar[rr]^-{d}& & \Omega^{2}_\clo(X)_\Z \ar[ru]^-{\mathrm{Rham}}\ar[rd]& \\
        0 \ar[ru]&&&&0
        }
    \end{align*}
Here $\Omega^n_\clo(X)_\Z$ denotes the closed forms of integral periods. 
The above diagram is commutative, and the diagonal sequences are exact. 

We would like to define $\widehat{H}^n(X; \Z)$ for higher $n$, having analogous structure maps and hexagons. 

\subsubsection{Cheeger-Simons' differential characters}

Here we explain a model for a differential extension of $H\Z$ given by Cheeger and Simons, in terms of {\it differential characters} \cite{CheegerSimonsDiffChar}. 

For a manifold $X$ and a nonnegative integer $n$, the group of {\it differential characters} $\widehat{H}_{\mathrm{CS}}^n(X; \Z)$ is the abelian group consisting of pairs $(\omega, k)$, where
\begin{itemize}
    \item A closed differential form $\omega \in \Omega_{\mathrm{clo}}^n(X)$, 
    \item A group homomorphism $k \colon Z_{\infty, n-1}(X; \Z) \to \R/\Z$, 
    \item $\omega$ and $k$ satisfy the following compatibility condition. 
    For any $c \in C_{\infty, n}(X; \Z)$ we have
    \begin{align}\label{eq_compatibility_diffchar}
        k(\del c) \equiv \langle \omega, c \rangle_X \pmod \Z. 
    \end{align}
\end{itemize}
Here $C_{\infty, n}(X; \Z)$ and $Z_{\infty, n-1}(X; \Z)$ denote the smooth singular chains and cycles with integer coefficients.

We have homomorphisms
\begin{align*}
R_{\mathrm{CS} } &\colon \widehat{H}_{\mathrm{CS}}^n(X; \Z) \to \Omega_{\mathrm{clo}}^n(X), \quad (\omega, k) \mapsto \omega \\
a_{\mathrm{CS}} &\colon \Omega^{n-1}(X)/\mathrm{Im}(d) \to  \widehat{H}_{\mathrm{CS}}^n(X; \Z) , \quad
    \alpha \mapsto (d\alpha, \alpha). 
\end{align*}
and the quotient map gives
\begin{align*}
    I_{\mathrm{CS}} \colon \widehat{H}_{\mathrm{CS}}^n(X; \Z) \to\widehat{H}_{\mathrm{CS}}^n(X; \Z) / \mathrm{Im}(a_{\mathrm{CS}}) \simeq  H^n(X; \Z). 
\end{align*}
% The quadruple $(\widehat{H}^*_{\mathrm{CS}}, R_{\mathrm{CS}}, I_{\mathrm{CS}} , a_{\mathrm{CS}} )$ is a differential extension of $H\Z$. 
  We get the {\it differential cohomology hexagon}, 
    \begin{align}\label{diag_hexagon}
        \xymatrix@R = 6pt@C=10pt{
        0 \ar[rd] & & & & 0 \\
        & H^{n-1}(X; \R/\Z)\ar[rr]^-{\mathrm{Bock}} \ar[rd] & & H^n(X; \Z) \ar[ru] \ar[rd]^-{\otimes \R}& \\
        H^{n-1}(X; \R)\ar[ru] \ar[rd] & & {\widehat{H}_\CS^n(X; \Z)}\ar[ru]^-{I} \ar[rd]^-{{R}}& & H^n(X; \R) \\
        & \Omega^{n-1}(X) / \Omega^{n-1}_\clo(X)_\Z \ar[ru]^-{{a}} \ar[rr]^-{d}& & \Omega^{n}_\clo(X)_\Z \ar[ru]^-{\mathrm{Rham}}\ar[rd]& \\
        0 \ar[ru]&&&&0
        }
    \end{align}
   
    The above diagram is commutative. 
    The diagonal sequences are exact. 

We can construct an isomorphism between the two models $\widehat{H}_{\mathrm{geom}}^2$ and $\widehat{H}_{\mathrm{CS}}^2$ by taking the holonomy of connections, 
\begin{align}
    \widehat{H}_{\mathrm{geom}}^2(X; \Z) \to \widehat{H}_{\mathrm{CS}}^2(X; \Z), \quad [L, \nabla] \mapsto (F_\nabla/(2\pi \sqrt{-1}), \mathrm{Hol}(L, \nabla)). 
\end{align}

\subsubsection{Hopkins-Singer's differential cocycles}
Let $X$ be a manifold. 
    An $n$-th {\it differential cocycle} on $X$ is an element 
    \begin{align*}
        (c, h, \omega) \in Z_\infty^n(X; \Z) \times C_\infty^{n-1}(X; \R) \times \Omega^n_\clo(X)
    \end{align*}
    such that
    \begin{align}
        \omega - c_\R = \delta h. 
    \end{align}
    Here $C^*_\infty$ and $Z^*_\infty$ denotes the groups of smooth singular cochains and cocycles. 
    We introduce the equivalence relation $\sim$ on differential cocycles by setting
    \begin{align*}
        (c, h, \omega) \sim (c+\delta b, h - b_\R - \delta k, \omega)
    \end{align*}
    for some $(b, k) \in C^{n-1}_\infty(X; \Z) \times C_\infty^{n-2}(X; \R)$. 
    
\begin{defn}[{$\widehat{H}^*_\HS(X; \Z)$ \cite{HopkinsSinger2005}}]
Set 
\begin{align*}
    \widehat{H}^n_\HS(X; \Z) := \{(c, h, \omega) : \mbox{differential }n\mbox{-cocycle on } X\} / \sim
\end{align*}
\end{defn}

The group $\widehat{H}^n_\HS(X; \Z)$ also fits into the hexagon \eqref{diag_hexagon}. 
In fact we have a natural isomorphism $\widehat{H}^n_\HS \simeq \widehat{H}^n_\CS$. 

\subsection{Differential $K$-theory}\label{subsec_FL}
$K$-theory is the most studied generalized cohomology theory, other than $H\Z$, in the context of differential cohomology. 
There are various models for topological $K$-theories. 
Correspondingly, there are various models for differential $K$-theory. 
Here we briefly review the model constructed by Freed and Lott \cite{FL2010} in terms of vector bundles with connections. 
We will explain another model in Subsection \ref{subsubsec_Kdiffchar} below. 

Recall that the topological $K$-theory $K^0(X)$ can be defined in terms of stable isomorphism classes $[E]$ of complex vector bundles over $X$. 
The differential refinement is, roughly speaking, given by equipping vector bundles with connections. 
An element of $\widehat{K}_\FL^0(X)$ is represented by a triple $(E, \nabla^E, \eta)$, where $(E, \nabla^E)$ is a hermitian vector bundle over $X$ with a unitary connection and $\eta \in \Omega^{2\Z -1}(X)/\mathrm{im}(d)$. 
We take equivalence classes with respect to 
a suitable equivalence relation including the relation
\begin{align}
    \left(E, \nabla_1^E, \eta\right) \sim \left(E, \nabla_0^E, \eta + \mathrm{CS}(\nabla_0^E, \nabla_1^E)\right), 
\end{align}
where $\mathrm{CS}(\nabla_0^E, \nabla_1^E)$ is the Chern-Simons form measuring the difference of two connections. 
We have the structure homomorphisms 
\begin{align}
    I &\colon \widehat{K}_\FL^0(X) \to K^0(X), \quad [E, \nabla^E, \eta] \mapsto [E]\\
    R &\colon \widehat{K}_\FL^0(X) \to \Omega_\clo^{2\Z}(X), \quad [E, \nabla^E, \eta] \mapsto \mathrm{Ch}(\nabla^E) + d\eta \\
    a &\colon \Omega^{2\Z -1}(X)/\mathrm{im}(d) \to \widehat{K}_\FL^0(X), \quad \eta \mapsto [0, 0, \eta]. 
\end{align}
We also have the similar hexagon as \eqref{diag_hexagon}. 

\subsection{The axiomatic formulation of differential cohomology}\label{subsec_axiom_diffcoh}

We explain the axiomatic formulation of differential cohomology theory by Bunke and Schick \cite{BSDiffKSurvey}. 
We remark that there are important another formulation in terms of sheaves of spectra \cite{BNVsheafofSpectra}. 

Let $E^*$ be a generalized cohomology theory. 
Let $N^\bullet$ be a graded vector space over $\R$ equipped with a transformation of cohomology theories
\begin{align}
    \mathrm{ch} \colon E^* \to H^*(-; N^\bullet). 
\end{align}
The universal choice is $N^\bullet = E^\bullet_\R(\pt)  =: V_E^\bullet$ with $\mathrm{ch}$ the Chern-Dold homomorphism (\cite[Chapter II, 7.13]{rudyak1998}) for $E$. 

For a manifold $X$, set $\Omega^*(X; N^\bullet) := C^\infty(X; \wedge T^*M \otimes_\R N^\bullet)$ with the $\Z$-grading by the total degree. 
Let $d \colon \Omega^*(X; N^\bullet) \to \Omega^{*+1}(X; N^\bullet)$ be the de Rham differential. 
We have the natural transformation
\begin{align*}
    \mathrm{Rham} \colon \Omega^*_{\mathrm{clo}}(X; N^\bullet) \to H^*(X; N^\bullet). 
\end{align*}

\begin{defn}[{Differential extensions of a cohomology theory, \cite[Definition 2.1]{BSDiffKSurvey}}]\label{def_diffcoh}
A {\it differential extension} of the pair $(E^*, \mathrm{ch})$ is a quadruple $(\widehat{E}, R, I, a)$, where
\begin{itemize}
    \item $\widehat{E}$ is a contravariant functor $\widehat{E} \colon \mathrm{Mfd}^\mathrm{op} \to \mathrm{Ab}^\Z$. 
    \item $R$, $I$ and $a$ are natural transformations
    \begin{align*}
        R &\colon \widehat{E}^* \to \Omega_{\mathrm{clo}}^*(-; N^\bullet) \\
        I &\colon \widehat{E}^* \to E^*\\
        a &\colon \Omega^{*-1}(-; N^\bullet) / \mathrm{im}(d) \to \widehat{E}^* . 
    \end{align*}
\end{itemize}
We require the following axioms. 
\begin{itemize}
    \item $R \circ a = d$. 
    \item $\mathrm{ch} \circ I = \mathrm{Rham} \circ R$. 
   \item For all manifolds $X$, the sequence
   \begin{align}\label{eq_exact_diffcoh}
       E^{*-1}(X) \xrightarrow{\mathrm{ch}} \Omega^{*-1}(M ; N^\bullet) / \mathrm{im}(d) \xrightarrow{a} \widehat{E}(X) \xrightarrow{I} E^*(X) \to 0
   \end{align}
   is exact. 
\end{itemize}
In the case $N^\bullet = V_E^\bullet$ and $\mathrm{ch}$ is the Chern-Dold homomorphism, we simply call it a {\it differential extension of $E^*$}. 
\end{defn}
Such a quadruple $(\widehat{E}, R, I, a)$ itself is also called a {\it generalized differential cohomology theory}. 
We usually abbreviate the notation and just write a generalized cohomology theory as $\widehat{E}^*$. 
The above axiom is equivalent to the half of the hexagon \ref{diag_hexagon}. 

Hopkins and Singer gave a model of an differential extension for an arbitrary $(E^*, \mathrm{ch})$ in terms of {\it differntial function spectra} (\cite{HopkinsSinger2005}). 
But we cannot guarantee that any differential cohomology theory is isomorphic to the Hopkins-Singer's model. 
For more about this uniqueness issue, we refer to \cite{BunkeSchick2010}.

\section{Background : Invertible QFTs and Anderson duals}\label{sec_QFT}

\subsection{QFTs and invertibility}
Basically, a quantum field theory (QFT) is a symmetric monoidal functor 
\begin{align}\label{eq_QFT_functor}
    \mathcal{T} \colon \Bord_{n}^{\mathcal{S}} \to \mathcal{C}. 
\end{align}
Here $n$ is the dimension of the theory, and $\mathcal{S}$ specifies a structure on manifolds, such as orientations, Riemannian metrics, spin structures and principal $G$-bundles with connections. 
The domain $\Bord_n^{\mathcal{S}}$ is the $n$-dimensional {\it Bordism category} of $\mathcal{S}$-manifolds. 
The mathematical formulation for it varies depending on the context. 
If we are talking about {\it non-extended topological QFTs (TQFTs)}, $\Bord_n^{\mathcal{S}}$ is an ordinary category whose objects are $(n-1)$-dimensional closed $\mathcal{S}$-manifolds and morphisms are given by bordisms by compact $n$-dimensional $\mathcal{S}$-manifolds (\cite{AtiyahTQFT}, \cite{SegalCFT}). 
In the case of {\it extended TQFTs}, $\Bord_n^{\mathcal{S}}$ is a higher category, up to $(\infty, n)$-category in the fully extended case (\cite{LurieTQFT}). 
If we consider possibly {\it non-topological} QFTs, we need to impose some smoothness to the functor \eqref{eq_QFT_functor}. 
For this we can use the smooth version of (higher) categories formulated by Grady and Pavlov \cite{GradyPavlovBordism}. 

According to the above variations, the formulation of the target $\mathcal{C}$ also depends on the context.  
Mathematically we can allow any symmetric monoidal ($(\infty, n)$-, smooth, etc) category. 
However, the physically natural target category is the category of $\Z/2$-graded complex vector spaces $\mathrm{sVect}_\C$ in the non-extended case. 
In the extended case, it is natural to consider higher categories extending $\mathrm{sVect}_\C$, such as $\mathrm{sAlg}_\C$. 
We do not know the canonical target in general. 
Moreover, in order for a functor to be physically meaningful, we need other conditions such as reflection positivity and unitarity.

A QFT $\mathcal{T}$ is called {\it invertible} if it factors as
\begin{align}\label{eq_functor_invertible}
     \mathcal{T} \colon \Bord_n^{\mathcal{S}} \to \mathcal{C}^\times \subset \mathcal{C}, 
\end{align}
where $\mathcal{C}^\times$ is the maximal Picard ($\infty$-, smooth...) groupoid of $\mathcal{C}$. 
Here a {\it Picard groupoid} is a symmetric monoidal category whose all objects and morphisms are invertible under the monoidal product and the composition, respectively. 
This means that $\mathcal{T}$ assigns something invertible to all objects and morphisms of $\Bord_n^{\mathcal{S}}$. 
For example, if $\mathcal{C} = \mathrm{sVect}_\C$, we have $\mathcal{C}^\times = \mathrm{sLine}_\C$, the groupoid of one dimensional $\Z/2$-graded complex vector spaces and grading-preserving invertible linear maps. 

Invertible QFTs are a special, but an important class of QFTs. 
For example they arise as {\it Symmetry Protected Topological phases (SPT phases)} in condensed matter physics. 
In that context, physical systems are described by Hamiltonians having uniquely gapped ground states, and the effective theories are considered as invertible theories. 
Invertible theories are also important in the study of {\it anomalies}. 
A $(n-1)$-dimensional anomalous theory is formulated as a boundary theory of a $n$-dimensional invertible theory. 
Thus the classification of anomalies in $(n-1)$-dimension reduces to the classification of invertible theories in $n$-dimension.

\subsection{Classification of invertible QFTs and Freed-Hopkins conjecture}\label{subsec_freedhopkins}
Fully extended invertible QFTs are believed to be classified in terms of generalized cohomology. 
The argument uses the {\it stable homotopy hypothesis}, which gives equivalence between Picard $\infty$-groupoids with connective spectra. 
By this the functor \eqref{eq_QFT_functor} is transformed into a map of spectra. This means that $\mathcal{T}$ is regarded as an element of generalized cohomology represented by the mapping spectra. 

For simplicity we focus on the case where the structure $\mathcal{S}$ is given by a tangential $G$-structure (for precise formulation see\cite[Section 3]{YamashitaYonekura2021} ), where $G= \{G_d, s_d, \rho_d\}_{d \in \Z_{\ge 0}}$ is a sequence of compact Lie groups with homomorphisms $s_d \colon G_d \to G_{d+1}$ and $\rho_d \colon G_d \to \mathrm{O}(d; \R)$ satisfying the compatibility with $\mathrm{O}(d; \R) \hookrightarrow \mathrm{O}(d+1, \R)$.  

By the work of Galatius, Madsen, Tillmann and Weiss \cite{GMTW}, the spectrum corresponding to the groupoid completion of $\Bord_n^G$ is the (shift of) {\it Madsen-Tillmann spectrum} $MTG$. 
This is a normal version of the Thom spectrum, and represents the tangential $G$-bordism homology theory $\Omega^G_*$. 

On the other hand, the spectrum corresponding to the target $\mathcal{C}^\times$ is unclear if we do not specify it. 
As explained above, we do not know the universal target of physically meaningful fully extended QFTs. 
Freed and Hopkins proposed an {\it ansatz} that the universal target for invertible {\it topological} QFTs corresponds to the spectrum $I\C^\times$, the {\it Brown-Comenetz dual} to the sphere (footnote \ref{footnote_BCdual}). 
Based on this ansatz, they conjecture that

\begin{conj}[{\cite[Conjecture 8.37]{Freed:2016rqq}}] \label{conj_intro}
There is a $1 : 1$ correspondence
\begin{align}\label{eq_conj_intro}
    \left\{
    \parbox{18em}{deformation classes of reflection positive invertible $n$-dimensional fully extended field theories with symmetry type\footnote{Here {\it symmetry types} of QFT's in \cite{Freed:2016rqq} are certain classes of $G$'s in this paper which satisfy an additional set of conditions. } $G$}
    \right\} \simeq (I\Omega^G)^{n+1}(\pt). 
\end{align}
\end{conj}

In fact they prove that {\it topological} ones are classified by the torsion part of $(I\Omega^G)^{n+1}(\pt)$. 

\subsection{Anderson duals}

In this subsection we collect basics on the Anderson duals for spectra. 
For more details, see for example \cite[Appendix B]{HopkinsSinger2005} and \cite[Appendix B]{FMS07}. 

First note that the assignment $X \mapsto \mathrm{Hom}(\pi_*(X), \R/\Z)$ for each spectrum $X$ satisfies the Eilenberg-Steenrood axioms, so it is represented by an spectrum denoted by $I(\R/\Z)$\footnote{
If we replace $\R/\Z$ by $\C^*$, we get the Brown-Comenetz dual to the sphere. \label{footnote_BCdual}
}.
On the other hand we have $H^n(X; \R) \simeq \mathrm{Hom}(\pi_n(X); \R)$. 
So we have a mod $\Z$ reduction map $H\R \to I(\R/\Z)$. 

The {\it Anderson dual to sphere $I\Z$} is the spectrum defined as the homotopy fiber
    \begin{align}
        I\Z \to H\R \to I(\R/\Z). 
    \end{align}
For a general spectrum $E$, we define its {\it Anderson dual} as the function spectrum, 
\begin{align}
    IE := F(E, I\Z). 
\end{align}
This implies that we have the exact sequence (Universal Coefficient Theorem)
\begin{align}\label{eq_exact_DE}
    0 \to \mathrm{Ext}(E_{n}(X), \Z) \to IE^{n+1}(X)  \to \mathrm{Hom}(E_{n+1}(X), \Z) \to 0. 
\end{align}

$MTG$ represents the $G$-bordism homology theory $\Omega^G_*$. 
The Anderson dual $I\Omega^G = F(MTG, I\Z)$ fits into the exact sequence
\begin{align}
    0 \to \mathrm{Ext}(\Omega^G_{n}(X), \Z) \to (I\Omega^G)^{n+1}(X)  \to \mathrm{Hom}(\Omega^G_{n+1}(X), \Z) \to 0. 
\end{align}

We note that $\mathrm{Ext}(\Omega^G_{n}(X), \Z)$ is the torsion part of $(I\Omega^G)^{n+1}(X)$ for reasonable spaces $X$. 
According to the Freed-Hopkins' result about the classification in {\it topological} case, this group classifies $n$-dimensional invertible {\it topological} QFTs on $G$-manifolds. 
Conjecture \ref{conj_intro} says that we can use whole of $(I\Omega^G)^{n+1}(X) $ to classify possibly non-topological invertible QFTs. 

In order to give a better understanding of this conjecture, Yonekura and the author gave a new model for $I\Omega^G$ with a physical interpretation (\cite{YamashitaYonekura2021}). 
It turns out that the natural object produced from QFTs is a differential refinement of it, and passing to topological group corresponds to taking deformation classes. 
Also it has turned out that the analogous construction applies to give a ``QFT-like'' model of differential refinement of $IE$ for general $E$, provided that a differential {\it homology} $\widehat{E}_*$ is given. 
In the rest of this article we explain this construction.

\section{Differential homology theories}\label{sec_diffhom}

In this section we give a formulation of differential {\it homology} theories and explain examples. 

\subsection{The axiom}\label{subsec_axiom_diffhom}

Here we give the axioms for differential extensions of generalized homology theories. 
It is just a straightforward homology-version of the axiom in \cite{BunkeSchick2010}. 
We set $V_\bullet^E := \pi_\bullet(E) \otimes_\Z \R$. 
We have the homology Chern-Dold character homomorphism
\begin{align}
    \mathrm{ch} \colon E_* \to H_*(-; V^E_\bullet). 
\end{align}
Here we focus on the differential extension in the coefficient $V^E_\bullet$ for our purpose, but the variant in coefficient can be formulated as in the cohomology case (Section \ref{subsec_axiom_diffcoh}). 

\begin{defn}[{Differential extensions of homology theories}]\label{def_diffhom}
A {\it differential extension} of a generalized homology theory $E_*$ is a quadruple $(\widehat{E}_*, R, I, a)$, where
\begin{itemize}
    \item $\widehat{E}_*$ is a covariant functor $\widehat{E}_* \colon \mathrm{Mfd} \to \mathrm{Ab}^\Z$. 
    \item $R$, $I$ and $a$ are natural transformations
    \begin{align*}
        R &\colon \widehat{E}_* \to \Omega^{\mathrm{clo}}_*(-; V_\bullet^E) \\
        I &\colon \widehat{E}_* \to E_*\\
        a &\colon \Omega_{*+1}(-; V_\bullet^E) / \mathrm{im}(\del) \to \widehat{E}_*.
        \end{align*}
\end{itemize}
\begin{itemize}
    \item $R \circ a = \del$. 
    \item $\mathrm{ch} \circ I = \mathrm{Rham} \circ R$. 
   \item For all manifolds $X$, the sequence
   \begin{align}\label{eq_exact_diffhom}
       E_{*+1}(X) \xrightarrow{\mathrm{ch}} \Omega_{*+1}(M ; V_\bullet^E) / \mathrm{im}(\del) \xrightarrow{a} \widehat{E}_*(X) \xrightarrow{I} E_*(X) \to 0
   \end{align}
   is exact. 
\end{itemize}
\end{defn}

\subsection{Examples}\label{subsec_ex_diffhom}

\subsubsection{Differential $K$-homology}\label{subsubsec_diff_Khom}
Elmrabty \cite{ElmrabtyDiffKhom} defines differential $K$-homology, by refining Baum-Douglas model of $K$-homology in terms of geometric cycles (\cite{BaumDouglasKhomology}). 
This has a physical interpretation in terms of {\it D-branes}. 

Recall that in the Baum-Douglas model, an element in $K_n(X)$ is represented by a {\it $K$-cycle}, a triple $(M, E, f)$ consisting of closed Spin$^c$-manifold $M$ with $\dim M \equiv n \pmod 2$, a complex vector bundle $E$ over $M$ and a continuous map $f \colon M \to X$. 
The $K$-homology group $K_n(X)$ is defined to be the group of equivalence classes of $K$-cycles under the equivalence relation generated by direct sum, bordism and ``vector bundle modification''. 

Correspondingly, the differential $K$-homology group $\widehat{K}_n(X)$ is represented by a {\it differential $K$-cycle} $((M, \nabla^M), (E, \nabla^E), f, \psi)$, where in addition to the data of a $K$-cycle, $M$ is equipped with a Spin$^c$-connection $\nabla^M$, $E$ is equipped with a hermitian metric and unitary connection $\nabla^E$, $f$ is required to be smooth, and we have additional data $\psi \in \Omega_{2\Z + n + 1}(X) / \mathrm{im}(\del)$. 
The group $\widehat{K}_n(X)$ is given by taking equivalence classes under the corresponding equivalence relation suitably modified with differential data. 
For example the bordism relation in this case is, if we have a $2\Z + n + 1$-dimensional compact $((W, \nabla^W), (\mathcal{E}, \nabla^\mathcal{E}), f, 0)$ with boundary, we set
\begin{align}\label{eq_bordism_Kcyc}
    \left( (\del W, \nabla^W|_{\del W}), (\mathcal{E}|_{\del W}, \nabla^\mathcal{E}|_{\del W}), f|_{\del W}, \int_W \mathrm{Todd}(\nabla^E) \mathrm{Ch}(\nabla^\mathcal{E})f^*
    \right) \sim 0. 
\end{align}

The axiom above is satisfied with the structure maps
\begin{align}
    R \left( [(M, \nabla^M), (E, \nabla^E), f, \psi]\right) &:= \int_M \mathrm{Todd}(\nabla^M)\mathrm{Ch}(\nabla^E)  f^* - \del \psi, \label{eq_R_Khom} \\
    I \left([(M, \nabla^M), (E, \nabla^E), f, \psi]\right)&:= [M, E, f], \\
    a  \left([\psi]\right)& := [\varnothing, 0, (\varnothing \to X), \psi]. 
\end{align}
where we identify $\Omega_n(X; V_\bullet^K) \simeq \Omega_{2\Z + n}(X)$. 

\subsubsection{Differential ordinary homology}\label{subsubsec_diff_HZ_homology}
We can construct a refinement $\widehat{H}^{\HS}_*(-; \Z)$ of the ordinary integral homology given by a homology-version of differential cocycles \cite{HopkinsSinger2005}. 
Namely, a {\it differential $n$-cycles} over $X$ is defined to be an element
\begin{align}
    (c, h, \omega) \in Z^\infty_n(X; \Z) \times C^\infty_{n+1}(X; \R) \times \Omega^{\mathrm{clo}}_n(X)
\end{align}
such that $\del h = \omega - c$ as smooth currents\footnote{
Note that we have $C_n^\infty \subset \Omega_n$. 
}. 
We define a homomorphism
\begin{align}
    \del \colon C_{n+1}^\infty(X; \Z) \times C^\infty_{n+2}(X; \R) \to Z^\infty_n(X; \Z) \times C^\infty_{n+1}(X; \R) \times \Omega^{\mathrm{clo}}_n(X)
\end{align}
by $\del(b, k) := (\del b, -\del k -b, 0)$. 
We define
\begin{align}
    \widehat{H}^{\HS}_n(X; \Z) := \{ \mbox{differential }n\mbox{-cocycles over }X\} / \mathrm{im} \del. 
\end{align}
It is straightforward to construct structure maps and check the axioms above.

\subsubsection{Differential $G$-bordism homology theory $\widehat{\Omega_*^G}$}\label{subsubsec_diff_bordism}
Here we construct a differential refinement of tangential $G$-bordism homology theories. 
We use the notations and conventions about differential $G$-structures introduced in \cite[Section 3]{YamashitaYonekura2021}. 
Let $G = \{G_d, s_d, \rho_d\}_{d \in \Z_{\ge 0}}$ be tangential structure groups. 
Then we get the {\it Madsen-Tillmann spectrum} $MTG$, which represents the {\it tangential $G$-bordism homology theory} $\Omega^G_*$. 
For details see \cite[Section 6.6]{Freed19}
The topological group $\Omega^G_n(X)$ is the stable $G$-bordism group, the group consisting of bordism classes $[M, g^{\rm{top}}, f]$, where $(M, g^{\rm{top}})$ is a $n$-dimensional closed manifold with a stable tangential $G$-structure and $f \colon M \to X$ is a continuous map. 

The differential refinement $\widehat{\Omega^G_n}(X)$ is constructed in terms of {\it differential stable tangential $G$-cycles}, as follows. 
A $n$-dimensional differential stable tangential $G$-cycle 
over $X$ is a triple $(M, g, f)$, where in this case $g$ is a {\it differential} stable $G$-structure (i.e., equipped with $G$-connection) and $f \colon M \to X$ is required to be smooth (for details see \cite[Section 3]{YamashitaYonekura2021}). 

We use the {\it Bordism Picard groupoid} $\hBord_n(X)$ defined in \cite[Definition 3.8]{YamashitaYonekura2021}. 
The objects are differential stable tangential $G$-cycles $(M, g, f)$ of dimension $n$ over $X$, and 
the morphisms are bordism class $[W, g_W, f_W]$ of bordisms of differential stable tangential $G$-cycles. 

We need the Chern-Weil construction in this setting (\cite[Subection 4.1.1]{YamashitaYonekura2021}) We set\footnote{In the general notation introduced in \eqref{eq_def_NE} below, we have $N^\bullet_G = N^\bullet_{MTG}$. We abbreviate the notation. } 
\begin{align}\label{eq_inv_poly}
    N_G^\bullet
    := H^\bullet(MTG; \R) = \varprojlim_{d} H^\bullet(BG_d; \R_{G_d}) = \varprojlim_{d}(\mathrm{Sym}^{\bullet/2}\mathfrak{g}_d^* \otimes_\R \R_{G_d})^{G_d}. 
\end{align}
In the case where $G$ is oriented, i.e., the image of $\rho_d$ lies in $\mathrm{SO}(d, \R)$ for each $d$, the $G_d$-module $\R_{G_d}$ is trivial and $N_G^\bullet$ is the projective limit of invariant polynomials on $\mathfrak{g}_d$. 
In general cases, $N_G^\bullet$ can be regarded as the projective limit of polynomials on $\mathfrak{g}_d$ which change the sign by the action of $G_d$.

A differential stable tangential $G$-structure $g$ on a manifold $M$ defines a homomorphism (\cite[Definition 4.4]{YamashitaYonekura2021}), 
\begin{align}
    \mathrm{cw}_g  \colon \Omega^*(M; N_G^\bullet) \to \Omega^*(M; \mathrm{Ori}(M)), 
\end{align}
where $\mathrm{Ori}(M)$ is the orientation line bundle of $M$. 

An object $(M, g, f)$ in $\hBord_n(X)$ gives a closed $n$-current
\begin{align}
    \cw(M, g, f) \in \Omega_n^{\mathrm{clo}}(X; V^{MTG}_\bullet) \subset \Hom_{\mathrm{conti}} (\Omega^{n}(X; N_G^\bullet), \R), 
\end{align}
by, for $\omega \in \Omega^{n}(X; N_G^\bullet)$, 
\begin{align}
    \cw(M, g, f)(\omega) := \int_M \cw_g(f^*\omega). 
\end{align}
Similarly, a bordism $(W, g_W, f_W) \colon (M_-, g_-, f_-) \to (M_+, g_+, f_+)$ of differential stable tangential $G$-cycles of dimension $n$ gives an $(n+1)$-current
\begin{align}\label{eq_cw_current_bordism}
    \cw(W, g_W, f_W) \in \Omega_{n+1}(X; V^{MTG}_\bullet) \subset \Hom_{\mathrm{conti}} (\Omega^{n+1}(X;  N_G^\bullet), \R), 
\end{align}
by, for $\omega \in \Omega^{n+1}(X;  N_G^\bullet)$, 
\begin{align}\label{eq_def_cw_omega_mor}
    \cw[W, g_W, f_W](\omega) := \int_W \cw_{g_W}({f_W}^*\omega), 
\end{align}
If we have two such bordisms $(W, g_W, f_W)$ and $(W', g_W', f_W')$ which are bordant, the corresponding currents \eqref{eq_cw_current_bordism} differs by an image of $\del$. 
Thus for a morphism $[W, g_W, f_W]$ in $\hBord_n(X)$ we get an element
\begin{align}
    \cw[W, g_W, f_W] \in \Omega_{n+1}(X; V^{MTG}_\bullet) / \mathrm{im}(\del). 
\end{align}

\begin{defn}[{$\widehat{\Omega_*^G}$}]\label{def_diff_bordism}
Let $X$ be a manifold and $n$ be an integer. 
\begin{enumerate}
    \item 
We set
\begin{align}
    \widehat{\Omega_n^G}(X) := \{(M, g, f, \eta)\} / \sim, 
\end{align}
where $(M, g, f)$ is an object in $\hBord_n(X)$ and $\eta \in \Omega_{n+1}(X; V^{MTG}_\bullet) / \mathrm{im}(\del)$. 
The relation $\sim$ is the equivalence relation generated by the relation
\begin{align}
    (M_-, g_-, f_-, \eta) \sim (M_+, g_+, f_+, \eta - \cw[W, g_W, f_W]) 
\end{align}
for each morphism $[W, g_W, f_W] \colon (M_-, g_-, f_-) \to (M_+, g_+, f_+)$ in $\hBord_n(X)$. 
We regard $\widehat{\Omega^G}$ as a functor $\mathrm{Mfd} \to \mathrm{Ab}^\Z$ in the obvious way. 
\item
We define natural transformations $R$, $I$ and $a$ by
\begin{align*}
    R &\colon \widehat{\Omega_n^G}(X) \to \Omega_n^{\mathrm{clo}}(X; V^{MTG}_\bullet) , \quad
    [M, g, f, \eta] \mapsto \cw(M, g, f) + \del \eta, \\
    I &\colon \widehat{\Omega_n^G}(X) \to \Omega_n^G(X), \quad
    [M, g, f, \eta] \mapsto [M, g, f], \\
    a  &\colon \Omega_{n+1}(X; V^{MTG}_\bullet) / \mathrm{im}(\del) \to \widehat{\Omega_n^G}(X) , \quad
    \eta \mapsto [\varnothing, \eta]. 
\end{align*}
\end{enumerate}
\end{defn}

We can easily check that the quadruple $(\widehat{\Omega_*^G}, R, I, a)$ satisfies the axiom in Definition \ref{def_diffhom}. 

\section{The Anderson duals to differential homology}\label{sec_IE}
\subsection{The construction}\label{subsec_def_hat_IE}
Let $E$ be a spectrum and assume that we are given a differential extension $(\widehat{E}_*, R_{E_*}, I_{E_*}, a_{E_*})$ of $E$-homology. 
In this section we explain that it associates a model $IE_\dR^*$ of the Anderson dual cohomology theory $IE$ and its differential extension of the pair $\left(IE^*, \mathrm{ch}'\right)$ (Definition \ref{def_diffcoh}), where $\mathrm{ch}'$ is defined below.  

Set
\begin{align}\label{eq_def_NE}
    N_E^\bullet := \Hom(\pi_{-\bullet} E, \R). 
\end{align}

By the third arrow in \eqref{eq_exact_DE}, we have a canonical transformation of cohomology theories
    \begin{align}\label{eq_ch'}
        \mathrm{ch}' \colon (IE)^*(X) \to H^*(X; N_E^\bullet) \simeq \Hom(E_*(X), \R), 
    \end{align}
    For example if $E_n(\pt)$ is finitely generated for all $n$, we have an isomorphism $V_{IE}^\bullet \simeq N_E^\bullet$ and \eqref{eq_ch'} coincides with the Chern-Dold homomorphism. 

We have the pairing
    \begin{align*}
        \langle -, -\rangle \colon \Omega_*(X; V_\bullet^E) \otimes \Omega^{*}(X; N_E^\bullet) \to \R. 
    \end{align*}
    Using this we regard $\Omega_*(X; V_\bullet^E) \subset  \Hom_{\mathrm{conti}} (\Omega^{*}(X; N_E^\bullet), \R)$.

\begin{defn}[{$(\widehat{IE}_\dR)^*$ and $IE^*_\dR$ associated to $\widehat{E}_*$}]\label{def_diff_IE}
Let $E$ be a spectrum and assume that we are given a differential extension $(\widehat{E}_*, R_{E_*}, I_{E_*}, a_{E_*})$ of $E$-homology. 
\begin{enumerate}
    \item 
For a manifold $X$ and integer $n$, we set
\begin{align}
    (\widehat{IE}_\dR)^n(X) := \{(\omega, h) \}, 
\end{align}
where
\begin{itemize}
    \item $\omega \in \Omega^n_{\mathrm{clo}}(X; N_E^\bullet)$. 
    \item $h \colon \widehat{E}_{n-1}(X) \to \R/\Z$ is a group homomorphism. 
    \item $\omega$ and $h$ fits into the following commutative diagram. 
    \begin{align}\label{diag_compatibility}
        \xymatrix{
        \Omega_n(X; V_\bullet^E) / \mathrm{im} (\del) \ar[rd]^-{a_{E_*}} \ar[rr]^-{(\mathrm{mod} \Z) \circ \langle -,  \omega \rangle}&& \R/\Z \\
        &\widehat{E}_{n-1}(X) \ar[ru]^-{h} &
        }. 
    \end{align}
\end{itemize}
\item
We define
\begin{align}
    a_{IE^*} \colon \Omega^{n-1}(X; N_E^\bullet) / \mathrm{im} (d) &\to (\widehat{IE}_\dR)^n(X), \quad
    \alpha \mapsto \left( 0, (\mathrm{mod}\Z )\circ \langle -, \omega \rangle \right). \notag
\end{align}
We set
\begin{align}
    IE^n_\dR(X) := (\widehat{IE}_\dR)^n(X)/\mathrm{im} (a_{IE^*}). 
\end{align}

\end{enumerate}
$(\widehat{IE}_\dR)^*$ and $IE^*_\dR$ are regarded as functors $\mathrm{Mfd}^{\mathrm{op}} \to \mathrm{Ab}^\Z$ in the obvious way. 
\end{defn}

\begin{defn}[{Structure maps for $(\widehat{IE}_{\mathrm{dR}})^*$ and $IE_{\mathrm{dR}}^*$}]\label{def_str_map_IE_dR}
We define the following maps natural in $X$. 
The well-definedness is easily checked. 
\begin{itemize}
    \item We denote the quotient map by
\begin{align*}
    I_{IE^*} \colon (\widehat{IE}_{\mathrm{dR}})^*(X) \to  IE_{\mathrm{dR}}^*(X). 
\end{align*}
\item We define
\begin{align*}
    R_{IE^*} \colon (\widehat{IE}_{\mathrm{dR}})^*(X) \to \Omega_{\mathrm{clo}}^*(X; N_E^\bullet) , \quad
    (\omega, h) \mapsto \omega. 
\end{align*}
\item We define
\begin{align*}
    \mathrm{ch}' \colon IE_{\mathrm{dR}}^*(X) \to H^n(X; N_E^\bullet) \left(\simeq \mathrm{Hom}(E_{n}(X), \R)\right), \quad
    I_{IE^*}((\omega, h)) \mapsto \mathrm{Rham}(\omega). 
\end{align*}
\item We define
\begin{align*}
    p \colon \mathrm{Hom}(E_{*-1}(X), \R/\Z) \to IE_{\mathrm{dR}}^*(X), \quad 
    h \mapsto I_{IE^*}((0, h\circ I_{E_*})). 
\end{align*}
\end{itemize}

\end{defn}

\subsection{Examples}\label{subsec_IE_ex}

Here we list some examples. 
We apply the construction in the last subsection to each of the examples of differential homology theories in Subsection \ref{subsec_ex_diffhom}. 
We will see that they recover the existing Cheeger-Simons type models of differential cohomology theories. 

\subsubsection{The differential ordinary cohomology in terms of differential characters by \cite{CheegerSimonsDiffChar}}\label{subsubsec_diffchar}
The ordinary integral cohomology is Anderson self-dual, $H\Z \simeq IH\Z$. 
Thus the above construction produces a differential extension of $H\Z$. 
Indeed, it is easy to verify that the above construction applied to $\widehat{H}^\HS(-; \Z)$ in Subsection \ref{subsubsec_diff_HZ_homology} recovers the Cheeger-Simons' model of differential ordinary cohomology in terms of differential characters.  

\subsubsection{The differential $K$-theory in terms of ``differential characters in $K$-theory'' by \cite{BenameurMaghfoul2006}}\label{subsubsec_Kdiffchar}
The $K$-theory is also Anderson self-dual, $K \simeq IK$. 
Thus the above construction applied to the geometric-cycle model of differential $K$-homology by \cite{ElmrabtyDiffKhom} mentioned in Subsection \ref{subsubsec_diff_HZ_homology} produces a differential $K$-theory in terms of functions which assign an $\R/\Z$-value to each differential $K$-cycle. 
Indeed, this recovers the ``differential characters in $K$-theory'' by Benameur and Maghfoul \cite{BenameurMaghfoul2006}. 

A typical element of $\widehat{IK}_\dR^0(X)$ in this model comes from a hermitian vector bundle $(F, \nabla^F)$ with a unitary connection over $X$. 
Indeed, we get a pair $\left(h_{(E, \nabla^E)}, \mathrm{Ch}(\nabla^E)\right) \in \widehat{IK}_\dR^0(X)$, where $h_{(E, \nabla^E)} \colon \widehat{K}_{-1}(X) \to \R/\Z$ is defined by
\begin{align}\label{eq_diffKelem}
    h_{(E, \nabla^E)}[(M, \nabla^M), (E, \nabla^E), f, \psi] := \overline{\eta}\left( D_{E \otimes f^*F} \right) + \langle f_* \psi, \mathrm{Ch}(\nabla^F)\rangle \pmod \Z. 
\end{align}
Here $D_{E \otimes f^*F}$ denotes the Spin$^c$ Dirac operator on $M$ twisted by $E \otimes f^*F$ with the connection $\nabla^E \otimes f^*\nabla^F$. 
$\overline{\eta}(D) = \frac{1}{2}(\eta(D) + \dim \ker D)$ is the reduced eta invariant. 
The well-definedness of the map \eqref{eq_diffKelem} uses the Atiyah-Patodi-Singer's index theorem. 
Indeed, to be compatible with the bordism relation \eqref{eq_bordism_Kcyc} of differential $K$-cycles, we need 
\begin{align}
    \int_W \mathrm{Todd}(\nabla^W) \mathrm{Ch}(\nabla^\mathcal{E}\otimes f^*\nabla^{F}) \equiv \overline{\eta} (D_{(\mathcal{E} \otimes f^*F)|_{\del W}}) \pmod \Z, 
\end{align}
which is a consequence of the APS index theorem. 

This element corresponds to the element $[E, \nabla^E, 0] \in \widehat{K}^0_\FL(X)$ in the Freed-Lott model (Subsection \ref{subsec_FL}). 
An easy generalization of the above construction gives the isomorphism $\widehat{K}^0_\FL(X) \simeq \widehat{IK}_\dR^0(X)$.

\subsubsection{The differential Anderson dual to $G$-bordism theories by \cite{YamashitaYonekura2021}}\label{subsubsec_IOmega}
Applying the construction for $\widehat{\Omega^G_*}$ in Subsection \ref{subsubsec_diff_bordism}, we recover the following model given in \cite{YamashitaYonekura2021}.
Recall that we are using the abbreviation $N_G^\bullet := N_{MTG}^\bullet$ \eqref{eq_inv_poly}. 

% We set 
% \begin{align}\label{eq_inv_poly}
%     N_G^\bullet &:= (V^{MTG}_\bullet)^\lor \\
%     &= H^\bullet(MTG; \R) = \varprojlim_{d} H^\bullet(BG_d; \R_{G_d}) = \varprojlim_{d}(\mathrm{Sym}^{\bullet/2}\mathfrak{g}_d^* \otimes_\R \R_{G_d})^{G_d}. 
% \end{align}
% In the case where $G$ is oriented, i.e., the image of $\rho_d$ lies in $\mathrm{SO}(d, \R)$ for each $d$, the $G_d$-module $\R_{G_d}$ is trivial and $N_G^\bullet$ is the projective limit of invariant polynomials on $\mathfrak{g}_d$. 
% In general cases, $N_G^\bullet$ can be regarded as the projective limit of polynomials on $\mathfrak{g}_d$ which change the sign by the action of $G_d$.

\begin{defn}[{$(\widehat{I\Omega^G_{\mathrm{dR}}})^*$ and $(I\Omega^G_{\mathrm{dR}})^*$, \cite{YamashitaYonekura2021}}]\label{def_hat_DOmegaG}
Let $n$ be a nonnegative integer. 
\begin{enumerate}
    \item 
Define $(\widehat{I\Omega^G_{\mathrm{dR}}})^n(X)$ to be an abelian group consisting of pairs $(\omega, h)$, such that
\begin{enumerate}
    \item $\omega$ is a closed $n$-form $\omega\in \Omega_{\mathrm{clo}}^n(X; N_G^\bullet)$. 
    \item $h$ is a group homomorphism
    $h \colon \mathcal{C}^{G_\nabla}_{n-1}(X) \to \R/\Z$. 
\item $\omega$ and $h$ satisfy the following compatibility condition. 
Assume that we are given two objects $(M_-, g_-, f_-)$ and $(M_+, g_+, f_+)$ in $h\Bord^{G_\nabla}_{n-1}(X)$ and a morphism $[W, g_W, f_W]$ from the former to the latter. 
Then we have
\begin{align}\label{eq_compatibility_dR}
    h([M_+, g_+, f_+]) - h([M_-, g_-, f_-]) = \mathrm{cw}(\omega)([W, g_W, f_W]) \pmod \Z, 
\end{align}
where the right hand side is defined in \eqref{eq_def_cw_omega_mor}. 
\end{enumerate}
Abelian group structure on $(\widehat{I\Omega^G_{\mathrm{dR}}})^n(X)$ is defined in the obvious way. 

\item
We define a homomorphsim of abelian groups, 
\begin{align}\label{eq_def_DOmega_a}
    a \colon \Omega^{n-1}(X; N_G^\bullet)/\mathrm{Im}(d) &\to  (\widehat{I\Omega^G_{\mathrm{dR}}})^n(X) \\
    \alpha &\mapsto (d\alpha, \mathrm{cw}(\alpha)). \notag
\end{align}
Here the homomorphism $\mathrm{cw}(\alpha) \colon \mathcal{C}^{G_\nabla}_{ n-1}(X) \to \R/\Z$ is defined by 
\begin{align}\label{eq_def_cw_alpha}
    \mathrm{cw}(\alpha) ([M, g, f]) := \int_M \mathrm{cw}_g(f^*\alpha) \pmod \Z. 
\end{align}
We set
\begin{align*}
    (I\Omega^G_{\mathrm{dR}})^n(X) := (\widehat{I\Omega^G_{\mathrm{dR}}})^n(X)/ \mathrm{Im}(a). 
\end{align*}

\end{enumerate}
\end{defn}

\subsection{The proof of $IE \simeq IE_\dR$}\label{subsec_proof_IE}

The goal of the rest of the section is to prove that the functor $IE_\dR^*$ is actually the model of the Anderson dual cohomology $IE$ to $E$. 
First we show that $IE_\dR$ fits into the exact sequence for the Anderson dual. 

\begin{prop}\label{prop_exact_IE_dR}
For any manifold $X$ and integer $n$, the following sequence is exact.  
    \begin{align}\label{eq_prop_exact_IE_dR}
         \mathrm{Hom}(E_{n-1}(X), \R)\to \mathrm{Hom}(E_{n-1}(X), \R/\Z) \xrightarrow{p} (IE_{\mathrm{dR}})^n(X)\\ \xrightarrow{\mathrm{ch}'}
        \mathrm{Hom}(E_n(X), \R)
        \to \mathrm{Hom}(E_{n}(X), \R/\Z). \notag
    \end{align}
\end{prop}

\begin{proof}
The proof is analogous to that for \cite[Proposition 4.25]{YamashitaYonekura2021}. 
\end{proof}

\begin{thm}\label{thm_IE_dR=IE}
There is a natural isomorphism of the functors $\mathrm{Mfd}^{\mathrm{op}} \to \mathrm{Ab}^\Z$, 
\begin{align*}
    F \colon IE_\dR \simeq IE, 
\end{align*}
which fits into the following commutative diagram. 
\begin{align}\label{diag_mainthm_IE}
    \xymatrix{
 0 \ar[r]& \mathrm{Ext}(E_{n-1}(-), \Z) \ar[r]^-p \ar@{=}[d] & IE_{\mathrm{dR}}^n \ar[r]^-{\mathrm{ch}} \ar[d]_-{\simeq}^-{F} & \mathrm{Hom}(E_n(-), \Z) \ar[r]\ar@{=}[d] & 0 \\
 0 \ar[r]& \mathrm{Ext}(E_{n-1}(-), \Z) \ar[r] & IE^n \ar[r] & \mathrm{Hom}(E_n(-), \Z) \ar[r] & 0
}
\end{align}
Here the bottom row is the exact sequence in \eqref{eq_exact_DE}.  
Moreover, the quintuple $((\widehat{IE}_\dR)^*, R, I, a)$ in Definitions \ref{def_diff_IE} is a differential extension of the pair $\left((IE)^*, \mathrm{ch}\right)$. 
\end{thm}

\begin{proof}
The proof is essentially the same as the corresponding claim for the case $E = MTG$ given in \cite[Theorem 4.51]{YamashitaYonekura2021}. 
We use the model of $IE^*$ in \cite[Corollary B.17]{HopkinsSinger2005} in terms of functors of Picard groupoids.

Given an element $(\omega, h) \in (\widehat{IE}_{\mathrm{dR}})^n(X)$, we get the associated functor of Picard groupoids, 
\begin{align}\label{eq_associated_functor}
        \widetilde{F}_{(\omega, h)} \colon \left(\Omega_n(X; V_\bullet^E) / \mathrm{im} (\del) \xrightarrow{a_{E_*}} \widehat{E}_{n-1}(X) \right) \to (\R \xrightarrow{\mathrm{mod} \Z} \R/\Z)
    \end{align}
by applying $h$ on objects and $\omega$ on morphisms. 
This is well-defined thanks to the commutativity of the diagram \eqref{diag_compatibility}. 
Moreover, given two elements $(\omega, h)$ and $(\omega', h')$, and an element $\alpha \in \Omega^{n-1}(X; (V_\bullet^E)^\lor)/\mathrm{Im}(d)$ so that $(\omega', h') - (\omega, h) = a_{IE^*}(\alpha)$, we get the associated natural transformation, 
\begin{align}\label{eq_associated_transformation}
    \widetilde{F}_\alpha \colon \widetilde{F}_{(\omega, h)} \Rightarrow \widetilde{F}_{(\omega', h')}, 
\end{align}
by $\widetilde{F}_\alpha  := \langle R_{E_*}(-), \alpha\rangle$. 
Summarizing, we have defined the homomorphism which is functorial in $X$,
\begin{align}\label{eq_functor_IE_dR}
    \widetilde{F}_X \colon IE_\dR^n(X) \to \pi_0  \mathrm{Fun}_{\mathrm{Pic}} \left(\left(\Omega_n(X; V_\bullet^E) / \mathrm{im} (\del) \xrightarrow{a_{E_*}} \widehat{E}_{n-1}(X) \right) , (\R \to \R/\Z) \right), 
\end{align}
where $\pi_0  \mathrm{Fun}_{\mathrm{Pic}}$ denotes the group of natural isomorphism classes of functors of Picard groupoids. 

Now by \cite[Corollary B.17]{HopkinsSinger2005} we have an isomorphism
\begin{align}\label{eq_proof_main_IE_0}
    IE^n(X) \simeq \pi_0\mathrm{Fun}_{\mathrm{Pic}} (\pi_{\le 1}(L(X^+\wedge E)_{1-n}), (\R \to \R/\Z)). 
\end{align}
Here we set $L$ is the spectrification and $\pi_{\le 1}$ is the fundamental Picard groupoid. 
We have $\pi_i(L(X^+\wedge E)_{1-n}) = E_{n-1+i}(X)$, and the exact sequence in the bottom row of \eqref{diag_mainthm_IE} becomes the canonical ones in this model (see the explanation in \cite[Fact 2.6]{YamashitaYonekura2021}). 

As mentioned in \cite[Example B.6]{HopkinsSinger2005}, the Picard groupoids $\mathcal{C}$ coming from maps of Abelian groups are equivalent to $(\pi_1(\mathcal{C}) \xrightarrow{0} \pi_0(\mathcal{C}))$, so we have an equivalence
\begin{align}\label{eq_proof_main_IE_1}
    \left(\Omega_n(X; V_\bullet^E) / \mathrm{im} (\del) \xrightarrow{a_{E_*}} \widehat{E}_{n-1}(X) \right)
    \simeq \left( \ker a_{E_*} \xrightarrow{0} \mathrm{coker} \ a_{E_*}
    \right),
\end{align}
and we have canonical isomorphisms 
\begin{align}
    \ker a_{E_*} &\simeq \mathrm{im}\left(E_n(X) \xrightarrow{\mathrm{ch}} H_n(X; V_\bullet^E)
    \right), \\
    \mathrm{coker}\ a_{E_*} &\simeq E_{n-1}(X), 
\end{align}
by the exactness of \eqref{eq_exact_diffhom}. 

Now we construct a (natural isomorphism class of) functor of Picard groupoids
\begin{align}\label{eq_proof_IE_3}
    \pi_{\le 1}(L(X^+\wedge E)_{1-n}) \to \left( \ker a_{E_*} \xrightarrow{0} \mathrm{coker} \ a_{E_*}
    \right)
\end{align}
as follows. 
Let $S\R/\Z$ denote the Moore spectrum for $\R/\Z$, so that we have $\pi_i(X^+ \wedge E \wedge S\R/\Z) = (E_{\R/\Z})_i(X)$. 
Let $(X^+ \wedge E \wedge S\R/\Z )\langle n \rangle \to X^+ \wedge E \wedge S\R/\Z$ denote the $n$-connected cover. 
Composing it with the map $\Sigma^{-1}(X^+ \wedge E \wedge S\R/\Z) \to X^+ \wedge E$ coming from the extension of coefficient group $0 \to \Z \to \R \to \R/\Z \to 0$, 
we get a morphism of spectra, 
\begin{align}\label{eq_proof_IE_4}
    \Sigma^{-1}\left((X^+ \wedge E \wedge S\R/\Z )\langle n \rangle \right) \to X^+ \wedge E. 
\end{align}
Let $J$ be the homotopy cofiber of the map \eqref{eq_proof_IE_4}. 
We have $\pi_{n}(J) \simeq H_n(X; V^E_\bullet)$ and $\pi_{n-1}(J) \simeq  E_{n-1}(X)$. 
Since $\pi_n(J)$ is torsion-free, the $k$-invariant for the Picard groupoid $\pi_{\le 1}(LJ_{1-n})$ vanishes. 
Thus we have
\begin{align}
    \pi_{\le 1}(LJ_{1-n}) \simeq \left(H_n(X; V^E_\bullet) \xrightarrow{0} E_{n-1}(X) \right).  
\end{align}
Compose this equivalence with the functor of fundamental Picard groupoids associated to the map $X^+ \wedge E \to J$. 
Then the image is contained in the subgroupoid $\left( \ker a_{E_*} \xrightarrow{0} \mathrm{coker} \ a_{E_*}
    \right)$. So we get the desired functor \eqref{eq_proof_IE_3}. 
By construction, the functor \eqref{eq_proof_IE_3} induces identity on $\pi_0$ and $\mathrm{ch}$ on $\pi_1$.

By composing \eqref{eq_functor_IE_dR} with the pre-composition of the functor \eqref{eq_proof_IE_3} and equivalence \eqref{eq_proof_main_IE_1}, we get the desired functor $F$ as
\begin{align}
    F_X \colon IE_\dR^n(X) \to \pi_0  \mathrm{Fun}_{\mathrm{Pic}} \left(\pi_{\le 1}(L(X^+\wedge E)_{1-n}) , (\R \to \R/\Z) \right) \xrightarrow[\simeq]{\eqref{eq_proof_main_IE_0}} IE^n(X). 
\end{align}

The commutativity of \eqref{diag_mainthm_IE} is obvious. 
Applying that diagram to any manifold $X$, we know that the top row is exact by Proposition \ref{prop_exact_IE_dR}, and bottom row is also exact. 
Thus by the five lemma we see that $F$ is a natural isomorphism. 
The last statement is easily checked. 
This completes the proof of Theorem \ref{thm_IE_dR=IE}. 

\end{proof}

\section{The interpretation of $\widehat{IE}_\dR$ in terms of QFTs}\label{sec_interpretation}

In this section we explain the interpretation of our model $\widehat{IE}_\dR$ in terms of QFTs. 

As explained in Subsection \ref{subsec_freedhopkins}, the Anderson dual to the $G$-bordism theory is supposed to classify deformation classes of possibly non-topological invertible QFTs on $G$-manifolds. 
Indeed, the $\R/\Z$-valued function $h$ of an element $(h, \omega) \in (\widehat{I\Omega^G_\dR})^{n+1}(X)$ can be regarded as the complex phase of the partition functions of 
an invertible QFTs for manifolds equipped with differential stable $G$-structures and smooth maps to $X$. 
The forgetful map $I \colon (\widehat{I\Omega^G_\dR})^{n+1}(X)\to (I\Omega^G_\dR)^{n+1}(X)$ is regarded as taking the deformation classes of such invertible QFTs. 

This interpretation is based on the following empirical facts known for physically meaningful\footnote{
Note that the physically meaningfulness includes conditions such as reflection positivity, locality and Wick-rotated unitarity. 
} invertible QFTs. 
Generally in a non-topological $n$-dimensional QFT $\mathcal{T}$, the partition function $Z_\mathcal{T}(M^n, \mathfrak{s}) \in \C$ for closed $n$-dimensional $\mathcal{S}$-manifold varies smoothly according to the smooth variation of the input $(M^n, \mathfrak{s})$. 
It is empirically known that, if $\mathcal{T}$ is {\it invertible}, the variation of the complex phase of $Z_\mathcal{T}$ can be measured by an integration of some characteristic form, i.e., there exists a $(d+1)$-dimensional characteristic polynomial $\omega_{\mathcal{T}}$ such that, for each bordism $(W^{d+1}, \mathfrak{s}) \colon (M_-, \mathfrak{s}_-) \to (M_+, \mathfrak{s}_+) $ we have
\begin{align}\label{eq_empirical1}
    \arg \left( \frac{Z_\mathcal{T}(M^d_+, \mathfrak{s}_+)}{Z_\mathcal{T}(M^d_-, \mathfrak{s}_-)} \right) = \int_W \cw(\omega_{\mathcal{T}})(W^{d+1}, \mathfrak{s}) \pmod \Z, 
\end{align}
where the left hand side is well-defined since the partition function is nonzero in the invertible case. 
Note that a priori there is nothing to assign for $(d+1)$-dimensional bordism, since $\mathcal{T}$ is $d$-dimensional theory. 

Moreoveer, it is also empirically known that the effect of 
smooth deformation of invertible theories can also be measured by local integrations. A smooth deformation $\mathcal{H}$ from $\mathcal{T}_0$ to $\mathcal{T}_1$ provides us a $d$-dimensional characteristic form $\alpha_\mathcal{H}$ so that
\begin{align}\label{eq_empirical2}
    \arg \left( \frac{Z_{\mathcal{T}_1}(M^d, \mathfrak{s})}{Z_{\mathcal{T}_0}(M^d, \mathfrak{s})} \right) = \int_M \cw(\alpha_\mathcal{H})(M^d, \mathfrak{s}) \pmod \Z.
\end{align}

These empirical facts explain our interpretation of our model. 
Indeed, recall that $\Omega^*(X; N_G^\bullet)$ is the differential forms on $X$ with coefficient in invariant polynomials on $\mathfrak{g}$, which can be regarded as characteristic polynomials for the structure $\mathcal{S}$ given by differential stable tangential $G$-structures and maps to $X$. 
The equation \eqref{eq_empirical1} corresponds to the compatibility condition \eqref{diag_compatibility} for $(\arg Z_\mathcal{T}, \omega_{\mathcal{T}})$, and the equation \eqref{eq_empirical2} says that the deformation corresponds to the addition by $a(\alpha_{\mathcal{H}})$ in \eqref{eq_def_DOmega_a}. 

Finally we comment on the possibility of interpreting the general construction of differential Anderson dual to differential homology $\widehat{E}_*$ also in terms of a kind of invertible QFTs. 
The idea is the following. 
Differential homology $\widehat{E}_*$ typically comes from a Picard groupoid $\mathcal{C}_\nabla$ with differential data. For example $h\mathrm{Bord}^{G_\nabla}_n(X)$, the groupoid of differential $K$-cycles (which has an interpretation in terms of {\it D-branes}), and the groupoid of differential ordinary cocycles. 
On morphisms of these categories, we can integrate differential forms. 
Differential Anderson dual $(\widehat{IE}_\dR)^*(X)$ classifies functors 
   \begin{align}
      (\omega, h) \colon \mathcal{C}_\nabla(X) \to (\R \to \R/\Z)
   \end{align}
which reflects the differential information. 
The map $h \colon \mathrm{Obj}(\mathcal{C}_\nabla) \to \R/\Z$ may be regarded as the phase of a partition function for some invertible QFTs with domain extending $\mathcal{C}_\nabla$.

   \bibliographystyle{ytamsalpha}
\bibliography{QFT}

\providecommand{\bysame}{\leavevmode\hbox to3em{\hrulefill}\thinspace}
\providecommand{\MR}{\relax\ifhmode\unskip\space\fi MR }
\providecommand{\MRhref}[2]{%
  \href{http://www.ams.org/mathscinet-getitem?mr=#1}{#2}
}
\providecommand{\href}[2]{#2}
\providecommand{\doihref}[2]{\href{#1}{#2}}
\providecommand{\arxivfont}{\tt}
\begin{thebibliography}{GTMW09}

\bibitem[Ati88]{AtiyahTQFT}
M.~Atiyah, \emph{Topological quantum field theories},
  \href{http://www.numdam.org/item?id=PMIHES_1988__68__175_0}{Inst. Hautes
  \'{E}tudes Sci. Publ. Math. (1988) 175--186 (1989)}.

\bibitem[BD82]{BaumDouglasKhomology}
P.~Baum and R.~G. Douglas, \emph{{$K$} homology and index theory}, Operator
  algebras and applications, {P}art 1 ({K}ingston, {O}nt., 1980), Proc. Sympos.
  Pure Math., vol.~38, Amer. Math. Soc., Providence, R.I., 1982, pp.~117--173.

\bibitem[BM06]{BenameurMaghfoul2006}
M.-T. Benameur and M.~Maghfoul, \emph{Differential characters in {$K$}-theory},
  \href{https://doi.org/10.1016/j.difgeo.2005.12.008}{Differential Geom. Appl.
  \textbf{24} (2006) 417--432}.

\bibitem[BNV16]{BNVsheafofSpectra}
U.~Bunke, T.~Nikolaus, and M.~V\"{o}lkl, \emph{Differential cohomology theories
  as sheaves of spectra}, \href{https://doi.org/10.1007/s40062-014-0092-5}{J.
  Homotopy Relat. Struct. \textbf{11} (2016) 1--66}.

\bibitem[BS10]{BunkeSchick2010}
U.~Bunke and T.~Schick, \emph{Uniqueness of smooth extensions of generalized
  cohomology theories}, \href{https://doi.org/10.1112/jtopol/jtq002}{J. Topol.
  \textbf{3} (2010) 110--156}.

\bibitem[BS12]{BSDiffKSurvey}
\bysame,
  \doihref{http://dx.doi.org/10.1007/978-3-642-22842-1\_11}{\emph{Differential
  {K}-theory: a survey}}, Global differential geometry, Springer Proc. Math.,
  vol.~17, Springer, Heidelberg, 2012, pp.~303--357.
  \url{https://doi.org/10.1007/978-3-642-22842-1_11}.

\bibitem[CS85]{CheegerSimonsDiffChar}
J.~Cheeger and J.~Simons,
  \doihref{http://dx.doi.org/10.1007/BFb0075216}{\emph{Differential characters
  and geometric invariants}}, Geometry and topology ({C}ollege {P}ark, {M}d.,
  1983/84), Lecture Notes in Math., vol. 1167, Springer, Berlin, 1985,
  pp.~50--80. \url{https://doi.org/10.1007/BFb0075216}.

\bibitem[Elm16]{ElmrabtyDiffKhom}
A.~Elmrabty, \emph{Differential {K}-homology and explicit isomorphisms between
  {$\Bbb R/\Bbb Z$}-{K}-homologies},
  \href{http://projecteuclid.org/euclid.bbms/1457560850}{Bull. Belg. Math. Soc.
  Simon Stevin \textbf{23} (2016) 1--19}.

\bibitem[FH21]{Freed:2016rqq}
D.~S. Freed and M.~J. Hopkins, \emph{Reflection positivity and invertible
  topological phases}, \href{https://doi.org/10.2140/gt.2021.25.1165}{Geom.
  Topol. \textbf{25} (2021) 1165--1330}.

\bibitem[FL10]{FL2010}
D.~S. Freed and J.~Lott, \emph{An index theorem in differential {$K$}-theory},
  \href{https://doi.org/10.2140/gt.2010.14.903}{Geom. Topol. \textbf{14} (2010)
  903--966}.

\bibitem[FMS07]{FMS07}
D.~S. Freed, G.~W. Moore, and G.~Segal, \emph{The uncertainty of fluxes},
  \href{https://doi.org/10.1007/s00220-006-0181-3}{Comm. Math. Phys.
  \textbf{271} (2007) 247--274}.

\bibitem[Fre19]{Freed19}
D.~S. Freed, \doihref{http://dx.doi.org/10.1090/cbms/133}{\emph{Lectures on
  field theory and topology}}, CBMS Regional Conference Series in Mathematics,
  vol. 133, American Mathematical Society, Providence, RI, 2019.
  \url{https://doi.org/10.1090/cbms/133}. Published for the Conference Board of
  the Mathematical Sciences.

\bibitem[GP20]{GradyPavlovBordism}
D.~Grady and D.~Pavlov,
  \doihref{http://dx.doi.org/10.48550/ARXIV.2011.01208}{\emph{Extended field
  theories are local and have classifying spaces}}, 2020.
  \url{https://arxiv.org/abs/2011.01208}.

\bibitem[GTMW09]{GMTW}
S.~r. Galatius, U.~Tillmann, I.~Madsen, and M.~Weiss, \emph{The homotopy type
  of the cobordism category},
  \href{https://doi.org/10.1007/s11511-009-0036-9}{Acta Math. \textbf{202}
  (2009) 195--239}.

\bibitem[HS05]{HopkinsSinger2005}
M.~J. Hopkins and I.~M. Singer, \emph{Quadratic functions in geometry,
  topology, and {M}-theory},
  \href{http://projecteuclid.org/euclid.jdg/1143642908}{J. Differential Geom.
  \textbf{70} (2005) 329--452}.

\bibitem[Lur09]{LurieTQFT}
J.~Lurie, \emph{On the classification of topological field theories}, Current
  developments in mathematics, 2008, Int. Press, Somerville, MA, 2009,
  pp.~129--280.

\bibitem[Rud98]{rudyak1998}
Y.~B. Rudyak, \emph{On {T}hom spectra, orientability, and cobordism}, Springer
  Monographs in Mathematics, Springer-Verlag, Berlin, 1998. With a foreword by
  Haynes Miller.

\bibitem[Seg04]{SegalCFT}
G.~Segal, \emph{The definition of conformal field theory}, Topology, geometry
  and quantum field theory, London Math. Soc. Lecture Note Ser., vol. 308,
  Cambridge Univ. Press, Cambridge, 2004, pp.~421--577.

\bibitem[YY21]{YamashitaYonekura2021}
M.~Yamashita and K.~Yonekura, \emph{{Differential models for the Anderson dual
  to bordism theories and invertible {QFT}'s, {I}}},
  \href{http://arxiv.org/abs/2106.09270}{{\arxivfont arXiv:2106.09270
  [math.AT]}}.

\end{thebibliography}

\end{document}